\documentclass[12pt]{amsart}
\usepackage{amssymb}
\usepackage{graphicx}

\newcommand{\aut}{\text{Aut}}

\newcommand{\bT}{{\mathbb T}}
\newcommand{\bZ}{{\mathbb Z}}

\newcommand{\cC}{{\mathcal C}}
\newcommand{\cL}{{\mathcal L}}
\newcommand{\cT}{{\mathcal T}}

\newtheorem{thm}{Theorem}%[section]
\newtheorem{prop}[thm]{Proposition}
\newtheorem{lem}[thm]{Lemma}
\newtheorem{cor}[thm]{Corollary}

\theoremstyle{remark}

\newtheorem{ex}[thm]{Example}

\theoremstyle{definition}
\newtheorem{defn}[thm]{Definition}

\title[Symplectized Torelli mapping tori]{Symplectized Torelli mapping tori}

\author{J.~Amor\'os}
\address{Departament de Matem\`atiques, Universitat
Polit\`ecnica de Catalunya, Diagonal 647, 08028 Barcelona, Spain}
\email{jaume.amoros@upc.edu}



\date{\today}

\begin{document}

\begin{abstract}
Examples of aspherical closed symplectic 4--manifolds are presented whose Sullivan minimal models are $(1,n)$--formal for any $n$, without being formal. They have as cohomology algebra, signature, canonical class, those of a product of surfaces $S_g \times \bT^2$, and the fundamental group, resp. $A_{\infty}$ category defined by their de Rham complex, is isomorphic to that of $S_g \times \bT^2$ up to brackets of order $n+1$, resp. products of order $n+1$. Nevertheless, the manifolds do not admit any holomorphic structure.

These examples are derived from the fact that the Torelli groups $\cT_g$ are pro--unipotent. The mapping tori for representatives of suitable classes in $\cT_g$ are considered, and their product with $S^1$ is symplectized \`a la Thurston.
\end{abstract}

\maketitle

{\em Dedicated to the memory of Agust\'{\i} Roig, who taught rational homotopy to the author, and to many others.}

\section*{Introduction}

Symplectic manifolds emerged in the last generation as a non--integrable version of holomorphic, or K\"ahler, manifolds, with prominent roles ranging from the qualitative study of Hamiltonian Mechanics to the topological classification of manifolds.Their comparison with complex algebraic and K\"ahler manifolds is a pertinent question, because the existence of a K\"ahler structure endows the manifold with analytic properties, from the flat Hermitian metric to the Laplacian identities in the Dolbeault complex, or pluriharmonicity of harmonic forms, which ultimately result in global, topological properties of the manifold such as formality, or the existence of a Hodge structure in cohomology which can be seen as a Dolbeault structure on the harmonic forms in the manifold.

Because of these reasons, there is an active quest for a notion of harmonic forms on a compact symplectic manifold with compatibility relations to the symplectic structure, such as the $d d^c$ Lemma for compact K\"ahler manifolds (\cite{Yan,Hal,CW}). Topological properties such as formality, Massey products or the Hard Lefschetz isomorphisms between cohomology groups of complementary degree provide context and limits to the notions from K\"ahler geometry that can be extended to the symplectic category (\cite{IRTU,FMU,UV}).

The aim of the present work is to present examples in order to shed light, from a topological and group--theoretical viewpoint, to this study of symplectic manifolds as a nonintegrable version of K\"ahler geometry. Examples are presented of aspherical, closed, symplectic 4--manifolds such that their cohomology algebra, signature, canonical class are those of a product of two compact Riemann surfaces $S_g \times E$, with $C_g$ of genus $g \ge 2$ and $E$ an elliptic curve, their associated $A_{\infty}$--category given by the de Rham complex is formal up to triple products, i.e. their Sullivan (1,2)--minimal model is formal, yet the manifolds themselves are not formal and do not admit any holomorphic structure.

Our examples are built by taking the mapping torus $M_\phi \to S^1$ of suitable diffeomorphisms $\phi : S_g \to S_g$ lying in the Torelli group of genus $g$, and multiplying it with $S^1$ to endow it with a symplectic structure. Building on properties of the lower central series of the surface groups $\pi_1(S_g,*)$, and on the pro--unipotence of the Torelli group, we propose an algorithm in Example \ref{ex:1nformals} to obtain examples of mapping tori which are, from a homotopy viewpoint, as close to the trivial family $S_g \times E$ as one wishes: their Sullivan $(1,n)$--minimal models are formal, and their Massey products of 1--cohomology classes up to length $n$ vanish up to any wished for value of $n$, yet the manifolds are not formal, and they will have a nonvanishing Massey product of some order $N>n$.   

This note is organized in the following way: section \ref{s:grups} discusses the required foundations on the lower central series of a group $\Gamma$, its nilpotent, $k$--unipotent completions for a field $k$ of characteristic zero, and Malcev algebras, and presents the fundamental groups and mapping class groups of surfaces with an orientation towards the lower central series in both cases. The only content in the section which the author has not found in the literature is the procedure in Example \ref{ex:1nformals}, which uses the pro--unipotent structure of the Torelli group and its lower central series to obtain explicit diffeomorphisms which are the identity in the fundamental group of the surface up to brackets of order $n$ for any prescribed $n$, but are claimed to be nontrivial in the Torelli group. Section \ref{s:mt} builds the symplectized mapping tori $X_\phi$ from the elements $\phi$ of the Torelli group, and deducts the $k$--homotopy type of $X_\phi$ from the effect of $\phi$ in the lower central series of $\pi_1(S_g,*)$. This yields the results of only partial formality, equivalently vanishing of Massey products up to a given length but not in all lengths, for the separating Dehn twist and for the diffeomorphisms proposed in Example \ref{ex:1nformals}, in the latter case provided that they are nontrivial as the author conjectures. This is applied to show that symplectized Torelli mapping tori do not admit a holomorphic structure, regardless of how small the difference between their homotopy type and that of the product of two Riemann surfaces is.

\noindent {\em Acknowledgements:} The author is grateful to Dieter Kotschick and Domingo Toledo for discussions over the topic of this work. The research has been partially supported by the Spanish State Research Agency
AEI/10.13039/50110 0 011033 grant PID2019-103849GB-I00 and by the AGAUR project 2021 SGR 00603 Geometry of Manifolds
and Applications, GEOMVAP.

\section{The mapping class and Torelli groups} \label{s:grups}

Let $S_g$ be an oriented, closed, connected surface of genus $g \ge 2$. From the homotopy viewpoint, it is an aspherical topological space with fundamental group given by a single relation
$$ 
\pi_1(S_g,*) = \langle a_1, \dots , a_g, b_1, \dots , b_g \; | \; (a_1,b_1) \dots (a_g,b_g) \rangle 
$$

The work of Dehn, Nielsen, Baer in group theory and low--dimensional, completed with the study by Whitney, Munkres and others of smooth maps (see \cite{Bir},\cite{FM} for more detailed accounts) allow for the presentation of the {\em mapping class group} of $S_g$ as
\begin{alignat*}{1}
M(g,0) &= \text{Out}^+ (\pi_1(S_g,*)) = \aut^+(\pi_1(S_g,*))/ \pi_1(S_g,*) \\
&\cong \text{Homeo}^+(S_g)/\text{Homeo}_0(S_g) \cong \text{Diff}^+(S_g)/\text{Diff}_0(S_g) \, ,
\end{alignat*}
where the action of the group $\pi_1(S_g,*)$ in its automorphism group is given by conjugation $g \cdot \varphi= g \varphi g^{-1}$, Homeo are homeomorphisms, Diff are diffeomorphisms of class $\cC^r$ for any $r=1, \dots, \infty$, the superscript $+$ means orientation--preserving, and the subscript $0$ means isotopic, resp. $\cC^r$--diffeotopic, to the identity.

\subsection{The nilpotent completion and the Malcev algebra}

Recall that for any group $\Gamma$, we can define the {\em bracket} of two of its elements as $(x,y)=xyx^{-1}y^{-1}$, and the {\em lower central series} of $\Gamma$ as the decreasing filtration
$$
\Gamma_1= \Gamma, \qquad \Gamma_n= \{ (x,y) | x \in \Gamma_{n-1}, y \in \Gamma \}
$$
We say that the group $\Gamma$ is nilpotent of class $n$, or $n$--step nilpotent, when $\Gamma_{n+1}= \{ 1 \}$.

The lower central series filtration associates to a group $ \Gamma$ several functorial objects, among which we recall:
\begin{itemize}
\item The abelianization $\Gamma/ \Gamma_2$, which is an Abelian group such that all group morphisms $\Gamma \to A$ with $A$ Abelian factor uniquely through the natural projection $\Gamma \to \Gamma/ \Gamma_2$.
\item The nilpotent quotients $\Gamma/ \Gamma_n$ for all $n \ge 2$, with nilpotency class $n$ and such that all group morphisms from $\Gamma$ to nilpotent groups of class $n$ factor through the natural projection $\Gamma \to \Gamma/ \Gamma_n$. These nilpotent quotients can be obtained by successive central extensions
\begin{equation} \label{eq:centrext}
0 \longrightarrow \Gamma_{n-1}/ \Gamma_n \longrightarrow \Gamma/ \Gamma_n \longrightarrow \Gamma/ \Gamma_{n-1} \longrightarrow 1
\end{equation}
and define the nilpotent completion of $\Gamma$ as the projective system of quotients
\begin{equation} \label{eq:proj}
\dots \rightarrow \Gamma/ \Gamma_{n} \rightarrow \Gamma/ \Gamma_{n-1} \rightarrow \dots \rightarrow \Gamma/ \Gamma_2
\end{equation}
\item The graded Lie algebra $\text{gr} \, (\Gamma) = \oplus_{n \ge 1} \Gamma_n/ \Gamma_{n+1}$, whose Lie bracket is induced by the group bracket in $\Gamma$, 
\item For a field $k$ with $\text{char}\,(k)=0$, the class $n$ $k$--unipotent completions $\Gamma/ \Gamma_n \otimes k$. These are defined iteratively through the central extensions of Eq. \ref{eq:centrext} by tensoring the Abelian kernel $\Gamma_{n-1}/ \Gamma_n$ and mapping the class in $H^2(\Gamma/ \Gamma_{n-1}; \Gamma_{n-1}/ \Gamma_n)$ defining the central extension to $H^2(\Gamma/ \Gamma_{n-1} \otimes k; \Gamma_{n-1}/ \Gamma_n \otimes k)$, where the former $H^2$ group sits as a lattice. The $\Gamma/ \Gamma_n \otimes k$ are $k$--unipotent matrix groups, containing the nilpotent quotients $\Gamma/ \Gamma_n$ modulo their torsion subgroups as lattices of maximal rank, and they form the $k$--unipotent completion of $\Gamma$, denoted $\Gamma \otimes k$ and defined by applying $\otimes k$ to the projective system of quotients \ref{eq:proj}. 
\item The Malcev algebra of the group $\Gamma$ over the field $k$, $\cL(\Gamma,k)$, is the projective system of nilpotent Lie algebras
\begin{equation} \label{eq:malcev}
\dots \rightarrow \cL_{n-1}(\Gamma,k) \rightarrow \cL_{n-2}(\Gamma,k) \rightarrow \dots \rightarrow \cL_1 (\Gamma,k) \cong \Gamma/ \Gamma_2 \otimes k
\end{equation}
defined by the projective system of the $k$--unipotent completion of $\Gamma$ through the categorical equivalence between unipotent Lie groups and nilpotent Lie algebras for a fixed field $k$ given by the exponential maps $\text{exp} : \cL_{n-1}(\Gamma,k) \rightarrow \Gamma/ \Gamma_n \otimes k$. The Lie bracket induces a lower central series filtration in the pro--unipotent Malcev algebra, inducing isomorphisms of graded Lie algebras
$$
\text{gr}\, \cL(\Gamma,k) \cong \text{gr}\, (\Gamma) \otimes k \cong \text{gr}\, (\Gamma \otimes k)
$$
\end{itemize}

See \cite{MKS} for the proof of these properties, and Appendix A of \cite{ABCKT} for their categorical presentation.

\medskip

The dual of each Malcev algebra $\cL_n(\Gamma,k)$ is Sullivan's $(1,n)$--minimal model (\cite{Sul,BG}). More precisely, if $\cL_n(\Gamma,k)$ is a finite--dimensional Lie algebra, and $V(1,n)= \text{Hom}\,(\cL,k)$ is its dual $k$--vector space, the dual map to the bracket 
$$ 
[.,.] : \wedge^2 \cL_n(\Gamma,k) \longrightarrow \cL_n(\Gamma,k)
$$
is a linear map
$$ 
d : V(1,n) \longrightarrow \wedge^2 V(1,n)
$$ 
such that $d^2=0$ by duality of the Jacobi identity in $\cL$. The free differential graded--commutative algebra (gcda) generated by $(V(1,n),d)$, denoted by $M(1,n)$, is the $(1,n)$--minimal model of the group $\Gamma$ (or of a topological space $X$ with $\pi_1(X,*) \cong \Gamma$). 
The projective system of $n$--step Malcev Lie algebras in \eqref{eq:malcev} dualizes as an injective system of free gcda given by inclusions
$$
M(1,1) \hookrightarrow M(1,2) \hookrightarrow \dots \hookrightarrow M(1,n) \hookrightarrow \dots
$$
whose injective limit, which we will denote $M(2,0)$, is Sullivan's 1--minimal model of $\Gamma$ (or of a topological space $X$ with $\Gamma \cong \pi_1(X,*)$).

A topological space $X$ has the property of formality when Sullivan's minimal model, of which the above described 1--minimal model is the start of the injective system, can be computed from the cohomology algebra $H^*(X; k)$ (\cite{DGMS}). The space is $(1,n)$--formal if the $(1,l)$--minimal models can be computed from $H^*(X; k)$ for $l \le n$, and 1--formal if it is $(1,n)$--formal for all $n$ (see \cite{ABCKT}).  

\bigskip

Due to their naturality, the automorphism group $\aut(\Gamma)$ acts on all of the above objects. The graded Lie algebra $\text{gr}\,(\Gamma)$ is generated by its degree 1 component $\Gamma/ \Gamma_2$. This means that each graded component $\Gamma_n/ \Gamma_{n+1} \otimes k$ is obtained as the quotient of the Lie elements in the tensor product $(\Gamma/ \Gamma_2 \otimes k)^{\otimes n}$ by relations induced by the defining presentation of $\Gamma$. Therefore, for each $\varphi \in \aut(\Gamma)$ the automorphism $\varphi_n/ \varphi_{n+1}$ that it induces in $\Gamma_n/ \Gamma_{n+1} \otimes k$ is the image of the automorphism $\varphi/ \varphi_2$ that it induces in $\Gamma/ \Gamma_2 \otimes k$ by the map 
$$ 
\aut_{k-lin}(\Gamma/ \Gamma_2 \otimes k) \longrightarrow \aut_{k-lin}(\Gamma_n/ \Gamma_{n+1} \otimes k) \subset (\Gamma/ \Gamma_2 \otimes k)^{\otimes n} \; ,
$$
where the inclusion of the arrival space is obtained by choosing a linear section to the quotient $\text{Lie} \left( (\Gamma/ \Gamma_2 \otimes k)^{\otimes n} \right) \to \Gamma_n/ \Gamma_{n+1} \otimes k$.

\subsection{The case of hyperbolic surface groups} \label{ss:hyp}

Henceforth, we shall examine these nilpotent completions in the case of hyperbolic ($g \ge 2$) surface groups. These groups admit a presentation with only one defining relation, 
\begin{equation} \label{eq:pres}
\pi_1(S_g,*)= \langle a_1, \dots, a_g, b_1, \dots, b_g | (a_1,b_1) \cdots (a_g,b_g) \rangle \; .
\end{equation}
They are residually nilpotent groups, i.e. $\cap_{n \ge 1} \Gamma_n = \{ 1 \}$, or equivalently the natural map to the nilpotent completion $\Gamma \rightarrow \lim_{\leftarrow} \Gamma/ \Gamma_n$ is injective, and have trivial center (see \cite{MKS,Lab}).
 
Labute (\cite{Lab}) found that their Malcev algebra $\cL(\pi_1(S_g,*))$ can be also defined as a quotient of the free Lie algebra generated by the $2g$ variables $A_1, \dots, B_g$ by the single, homogeneous in degree 2, relation $[A_1,B_1]+ \dots +[A_g,B_g]=0$. This implies that the Malcev algebra is isomorphic to the graded Lie algebra $\text{gr} (\pi_1(S_g,*))$, and is residually nilpotent. The subalgebra formed by omitting just any of the generators is free, so the Malcev algebra has a trivial center.

Hyperbolic surface groups have abelianizations $\pi_1(S_g,*)/ \pi_1(S_g,*)_2 \cong H_1(S_g; \bZ) \cong \bZ^{2g}$. For any field $k$ with $\text{char} \, (k)=0$, the intersection alternate 2-form in the surface homology $H_1(S_g; k)$ is preserved by automorphisms of its fundamental group, so the action of the automorphism group defines a morphism
\begin{equation} \label{eq:ss}
\aut(\pi_1(S_g,*)) \longrightarrow \text{Sp}\,(2g,k)
\end{equation}
This morphism is onto (see \cite{MKS}), and it factors through the quotient $\aut(\pi_1(S_g,*)) \to M(g,0)$ because the conjugation action of $\pi_1(S_g,*)$ on its abelianization is trivial. 

For any finitely presented group $\Gamma$, the action of $\aut(\Gamma)$ on the Malcev algebra $\cL(\Gamma,k)$ preserves the lower central series filtration $\{ \cL(\Gamma,k)_n \}_{n \ge 1}$. The induced action in its graded pieces $\Gamma_n/ \Gamma_{n+1} \otimes k$ factors through the morphism in \eqref{eq:ss}. Fix a nilpotency class $n$, and a linear basis for the Malcev algebra $\cL_n(\Gamma,k)$ given by a linear isomorphism with its graded quotient $\oplus_{l=1}^{n} \Gamma_l/ \Gamma_{l+1} \otimes k$. For any $\varphi \in \aut(\Gamma)$, the induced automorphism $\varphi_* : \cL_n(\Gamma,k) \to \cL(\Gamma,k)$ has a box structure arising from the graded quotient for which the matrix is block--lower triangular. The restriction of the action of $\aut(\Gamma)$ to the diagonal blocks factors through its action on the first block $\Gamma/ \Gamma_2 \otimes k$ and this diagonal block action is actually defined on the outer automorphism group $\text{Out}\,(\Gamma)=\aut(\Gamma)/\Gamma$. 

In the case of $\Gamma= \pi_1(S_g,*)$, the action of its automorphisms on the abelianization is semisimple, so for these groups we can interpret the action of $\aut\,(\Gamma)$ on the Malcev algebras $\cL_n(\Gamma,k)$ as a unipotent extension of the semisimple action on the abelianization. In the projective limit, the action of $\aut\,(\Gamma)$ on the Malcev algebra $\cL(\Gamma,k)$ is a pro--unipotent extension of its action on the abelianization $\Gamma/ \Gamma_2 \otimes k$.

\subsection{The Torelli group}

After studying the actions of $\aut\,(\Gamma), \; \text{Out}\,(\Gamma)$ in the Malcev algebra $\cL(\Gamma,k)$ for hyperbolic surface groups we are ready to introduce the group of our main concern:

\begin{defn}
Let $S_g$ be an orientable, closed, connected surface of genus $g \ge 2$, $\Gamma= \pi_1(S_g,*)$. 
\begin{enumerate}
\item The Torelli group of genus $g$ with one marked point, $\cT_{g,1}$, is the kernel of the morphism
$$ 
\aut\,(\Gamma) \longrightarrow \text{Sp}\,(\Gamma/ \Gamma_2)
$$
\item The Torelli group of genus $g$, $\cT_g$, is the kernel of the morphism 
$$ 
\text{Out}\,(\Gamma) \longrightarrow \text{Sp}\,(\Gamma/ \Gamma_2)
$$
\end{enumerate}
\end{defn}

Here $\text{Sp}\,(\Gamma/ \Gamma_2)$ denotes the group of automorphisms of $\Gamma/ \Gamma_2 \cong \bZ^{2g}$ which preserve the nondegenerate, alternating intersection pairing $\wedge^2 H_1(S_g; \bZ) \to H_0(S_g; \bZ) \cong \bZ$.

As inner automorphisms of a group act trivially on its abelianization, the two presented Torelli groups are related by the exact sequence
$$
1 \longrightarrow \pi_1(S_g,*) \longrightarrow \cT_{g,1} \longrightarrow \cT_g \longrightarrow 1
$$
We will use the groups $\cT_{g,1}$ in an ancillary role, based in two of their properties: the first is that, unlike in the case of Torelli groups, there is an action
\begin{equation} \label{eq:preT}
\cT_{g,1} \longrightarrow \aut\,(\cL(\Gamma,k))
\end{equation}
which is the restriction of the action of $\aut\,(\Gamma)$ discussed above, and the projective limit of the actions of $\cT_{g,1}$ on the $n$--step Malcev algebras $\cL_n(\Gamma,k)$. The second property is:

\begin{thm} \label{t:asada} (Magnus, Asada \cite{Asa})
The action \ref{eq:preT} is faithful.
\end{thm}

We can restate this theorem as saying that the Torelli group with one marked point is pro--unipotent, i.e. a subgroup of the projective limit of the $k$--unipotent groups $\aut\,(\cL_n(\Gamma,k))$. Inmediate consequences are:

\begin{cor} (Magnus, Asada \cite{Asa})
\begin{enumerate}
\item $\cT_{g,1}$ does not have torsion.
\item $\cT_{g,1}$ is residually nilpotent.
\item $\cT_{g,1}$ is residually finite.
\end{enumerate}
\end{cor}

\subsection{Elements of interest in the Torelli group}

The Torelli group is generated by two kinds of self--diffeomorphisms of $S_g$ (\cite{Pow}):
\begin{enumerate}
\item {\em Johnson twists}: one of these is obtained by taking two disjoint smooth simple closed curves $C_1,C_2$ such that $S_g \setminus C_i$ is connected for $i=1,2$, but $S_g \setminus (C_1 \cup C_2)$ has two connected components $\Sigma, \tilde{\Sigma}$. Orient $C_1,C_2$ so that with the orientation of $S_g$ we have that $\partial \Sigma= C_1 + C_2$. The Johnson twist is the composition of two Dehn twists along $C_1,C_2$, with support in disjoint bicollar neighbourhoods of the two curves. The orientation condition means that $[C_1]+[C_2]=0 \in H_1(S_g;\bZ)$, so by the classical Picard--Lefschetz formula (see \cite{AVGZ}, \cite{Amo}, \cite{AMO}) the map induces the identity in $H_1(S_g; \bZ)$.  
\item {\em Separating Dehn twists}: these are Dehn twists along a smooth simple closed curve $C$ which disconnects $S_g$. As $C$ is the boundary of each component of $S_g \setminus C$, it is nullhomologous and by the classical Picard--Lefschetz formula the Dehn twist induces the identity in homology. We will see below that it is not the identity in homotopy if $C$ is not contractible.
\end{enumerate} 

Let us study with more detail the separating Dehn twist along a smooth curve $C$:
the homotopy (or non--abelian) Picard--Lefschetz formula (\cite{AMO}, see also Ch. 3 of \cite{Amo}) states that the image of every closed loop $\gamma$ intersecting $C$ transversely is homotopic to the insertion of a copy of $C$, oriented according to the sign of the intersection, at each intersection point in $C \cap \gamma$. 

%FIGURA PICARD-LEFSCHETZ HOMOTOPIC ??
\begin{figure}[ht]
\includegraphics[scale=0.6]{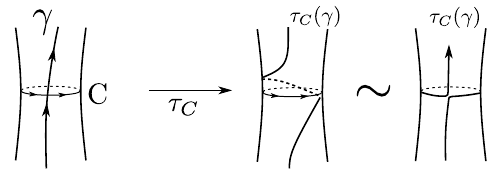}
\caption{Homotopy Picard--Lefschetz formula: up to homotopy, the effect of a Dehn twist along a simple loop $C$ on a path $\gamma$ intersecting $C$ transversely is to insert a copy of $C$ at each intersection point with it of $\gamma$.}
%\label{f:2reglatges}
\end{figure}

For $C$ to be homotopically nontrivial, it must decompose $S_g \setminus C$ as a disjoint union $\Sigma_1 \sqcup \Sigma_2$, where each $\Sigma_i$ is an oriented surface of genus $g_i>0$ and one boundary component. We can choose a presentation of $\pi_1(S_g,*)$ adapted to this decomposition by choosing a base point $* \in \Sigma_1$, loops $a_1, b_1, \dots, a_{g_1}, b_{g_1}$ generating $\pi_1(\Sigma_1,*)$ so that the conjugation class of the boundary $C$ is $w_C=(a_1,b_1) \cdots (a_{g_1},b_{g_1})$, a base point $\tilde{*}$ and loops $\tilde{a}_1, \tilde{b}_1, \dots, \tilde{a}_{g_2}, \tilde{b}_{g_2}$ generating $\pi_1(\Sigma_2,\tilde{*})$ so that the conjugation class of the boundary $C$ is $(\tilde{a}_1,\tilde{b}_1) \cdots (\tilde{a}_{g_2},\tilde{b}_{g_2})$, and finally replacing the generators from $\Sigma_2$ with loops $a_{g_1+i}=\gamma \cdot \tilde{a}_i \cdot \gamma^{-1}, b_{g_1+i}=\gamma \cdot \tilde{b}_i \cdot \gamma^{-1},$ where $\gamma$ is a smooth, simple path from $*$ to $\tilde{*}$ intersecting $C$ transversely at one point (see Fig. \ref{f:pradapt}).

%FIGURA PRADAPT: PRESENTACIO ADAPTADA ??
\begin{figure}[ht]
\includegraphics[scale=0.45]{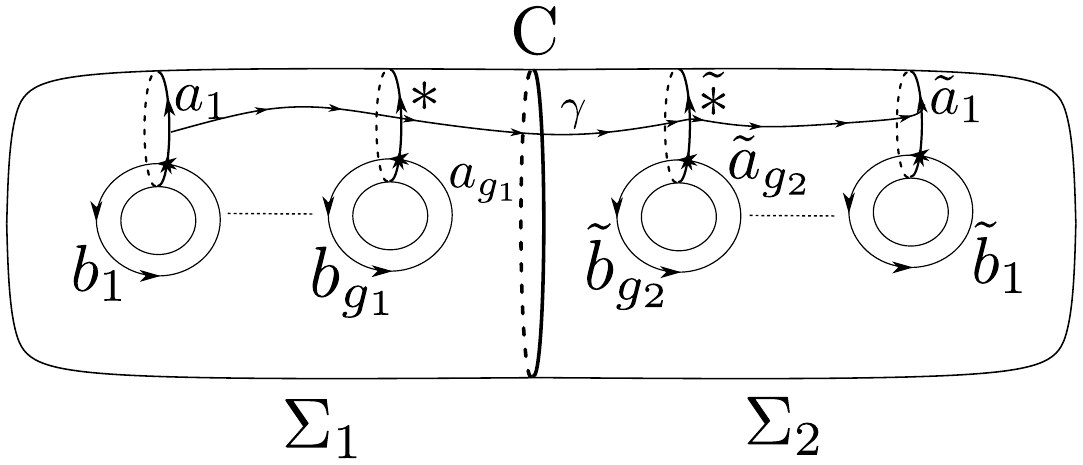}
\caption{Presentation of $\pi_1(S_g,*)$ adapted to a decomposition by a simple closed curve $C$ in 2 connected components with genus $g_1,g_2>0$.}
\label{f:pradapt}
\end{figure}

By the homotopy Picard--Lefschetz formula, a Dehn twist $\tau_C$ along the separating curve $C$ leaves invariant the loops $a_1,b_1, \dots , a_{g_1}, b_{g_1}$ in $\Sigma_1$ as they are disjoint from $C$. The loops $a_{g_1+1},b_{g_1+1}, \dots , a_{g}, b_{g}$ intersect $C$ twice, both times at the point $C \cap \gamma$ as the path $\gamma$ is covered on the way from $*$ to $\tilde{*}$ and back. Thus the image of each of these loops is 
$$
\tau_{C*} (a_j)= w_C a_j w_C^{-1}=(w_C,a_j) a_j \, , \qquad \tau_{C*} (b_j)= w_C b_j w_C^{-1}=(w_C,b_j) b_j 
$$
We have thus found, with the adapted basepoint $*$ and presentation of $\pi_1(S_g,*)$, that the separating Dehn twist is the identity modulo $\pi_1(S_g,*)_3$. Each of the group elements $\tau_{C*}(a_j)a_j^{-1},\tau_{C*}(b_j) b_j^{-1}$ belongs to the basis of $\pi_1(S_g,*)_3/ \pi_1(S_g,*)_4$ found by Labute in \cite{Lab}, so $\tau_{C*}$ is not the identity modulo $\pi_1(S_g,*)_4$. 

We have recalled at the end Section \ref{ss:hyp} how the preservation of the lower central series filtration induces a block structure in the matrices of the action $\aut\,(\Gamma) \rightarrow GL(\cL_n(\Gamma,k))$ for which they are block--lower triangular, and the diagonal blocks, i.e. the induced automorphisms in $\Gamma_l/ \Gamma_{l+1} \otimes k$, factor through the induced automorphism in $\Gamma/ \Gamma_2 \otimes k \cong H_1(\Gamma; k)$. Torelli automorphisms of $\pi_1(S_g,*)$ are the identity in $H_1(S_g; k)$, so the pro--unipotence of the Torelli group manifests itself as a projective limit of morphisms $\cT_{g,1} \rightarrow GL(\cL_n(\Gamma,k))$ whose images are block--unipotent (lower triangular) matrices. An immediate consequence of this structure, which can also be proved directly using brackets in the fundamental group, is

\begin{lem} \label{l:ordre}
If $\tau_1, \tau_2$ are two elements in $\cT_{g,1}$ such that they are the identity in $\pi_1(S_g,*)/ \pi_1(S_g,*)_{l_1}, \pi_1(S_g,*)/ \pi_1(S_g,*)_{l_2}$ respectively, then the bracket $(\tau_1,\tau_2)=\tau_1 \tau_2 \tau_1^{-1} \tau_2^{-1}$ is the identity in $\pi_1(S_g,*)/ \pi_1(S_g,*)_{l_1+l_2-1}$.   
\end{lem}

This property lets us find elements of the Torelli group with increasingly subtler effect on $\pi_1(S_g,*)$:

\begin{ex} \label{ex:1nformals}
Let $S_4$ be an oriented, closed, connected surface of genus 4, and $C_1,C_2 \subset S_4$ two simple closed curves, such that $S_4 \setminus C_i$ has two connected components of genus 2 for $i=1,2$, $S_4 \setminus (C_1 \cup C_2)$ has four connected components of genus 1 with connected boundary, $C_1,C_2$ intersect transversely at 2 points (see Fig. \ref{f:4anses}).

%FIG DESCOMPOSICIO 4 ANSES
\begin{figure}[ht]
\includegraphics[scale=0.6]{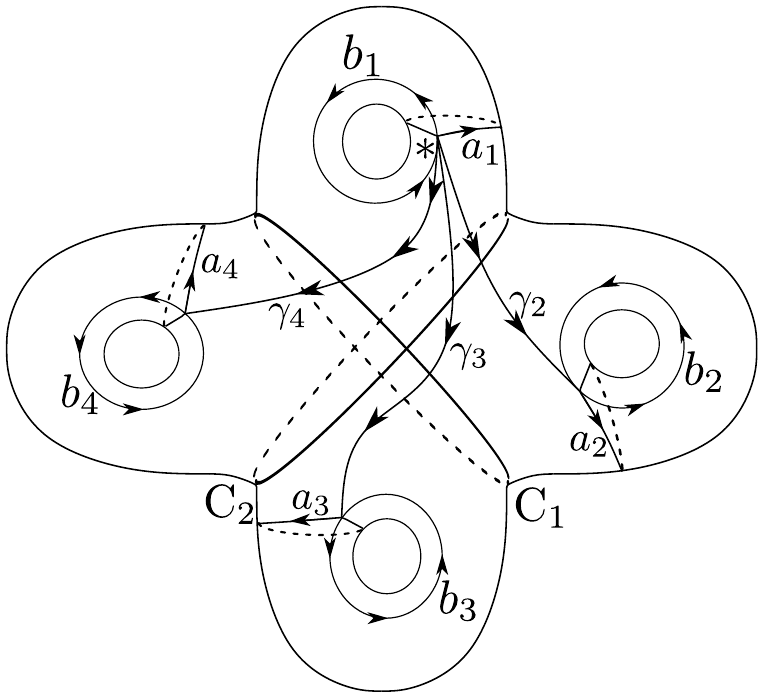}
\caption{Decomposition of a genus 4 surface $S_4$ in 4 handles of genus 1 by two simple closed curves $C_1,C_2$ meeting transversely at 2 points, and an adapted presentation for $\pi_1(S_4,*)$, in which $C_1=(a_1,b_1)(a_2,b_2), C_2=(a_1,b_1)(a_4,b_4)$. Note that $\tau_{C_1}, \tau_{C_2}$ both fix the classes $a_1,b_1$ of the handle containing the base point $*$, and that for every other handle with generators $a_i,b_i$ the effect of $\tau_{C_j}$ is the identity if the path $\gamma_i$ connecting the handle to the base point $*$ does not cross $C_j$, and conjugation by $C_j$ if the path $\gamma_i$ does cross $C_j$.}
\label{f:4anses}
\end{figure}

The separating Dehn twists $\tau_{C_1}, \tau_{C_2}$ are the identity modulo $\pi_1(S_4,*)_3$. Therefore, by Lemma \ref{l:ordre}, the sequence of brackets 
$$
(\tau_{C_1}, \tau_{C_2}), \quad (\tau_{C_1},(\tau_{C_1}, \tau_{C_2})), \quad (\tau_{C_1},(\tau_{C_1},(\tau_{C_1}, \tau_{C_2}))), \dots
$$
are the identity modulo $\pi_1(S_4,*)_l$ for $l=5,7,9,\dots$. The author conjectures that this is the maximal order of brackets for which they are the identity, even in $\cT_g$. Indeed, it is likely that $\tau_{C_1}, \tau_{C_2}$ generate a free subgroup of the Torelli group, thus any nontrivial bracket of length $l$ defined by iterative use of the two generators should be the identity modulo $\pi_1(S_4,*)_{2l+1}$, but not modulo $\pi_1(S_4,*)_{2l+2}$.
\end{ex}

\section{Symplectized mapping tori} \label{s:mt}

Let now $\phi : S_g \rightarrow S_g$ be an orientation--preserving diffeomorphism. Its {\em mapping torus} is the quotient topological space defined by the equivalence relation
$$
M_\phi= S_g \times [0,1]/ (x,0) \sim (\phi(x),1)
$$
for all $x \in S_g$. Classically known properties of these mapping tori include:

\begin{prop} \label{p:mt} 
\begin{enumerate}
\item $M_\phi$ is a closed, connected 3--manifold, endowed with a fibration $p : M_\phi \to S^1$ given by projection on the second factor before applying the equivalence relation, and oriented by the orientation of $S_g$ and the canonical orientation of $S^1$.
\item The diffeomorphism class of $M_\phi$ depends only on the difeotopy class of $\phi$. 
\item If we define the subspaces of vanishing cycles $Ev$ and of invariant cycles $Inv$ in $H_1(S_g; \bZ)$ by the exact sequence
$$
0 \longrightarrow Inv \longrightarrow H_1(S_g; \bZ) \stackrel{\phi_*-\text{Id}}{\longrightarrow} H_1(S_g; \bZ) \longrightarrow Ev \longrightarrow 0
$$
we have isomorphisms
\begin{alignat*}{1}
H_1(M_\phi; \bZ) &\cong H_1(S_g)/Ev \oplus \bZ [\sigma(S^1)] \\
H_2(M_\phi; \bZ) &\cong Inv \oplus \bZ [S_g]
\end{alignat*}
where $\sigma : S^1 \to M_\phi$ is a smooth section, and $[.]$ denotes the fundamental class in homology. The isomorphism in degree 2 is realized by the parallel transport of invariant cycles: for $c \in Inv$,
$$
tr(c)=[c \times [0,1]]+e \, ,
$$
where the first summand is the projection to the mapping torus of the fundamental class of $c \times [0,1] \subset S_g \times [0,1]$, and $e$ is a 2-chain in $S_g$ with boundary $de=\phi_*(c)-c$.
\item The fundamental group of the mapping torus $M_\phi$ is an HNN extension of the fundamental group of $S_g$, admitting a presentation
\begin{alignat*}{1}
\pi_1(M_\phi,*)= \langle & a_1,b_1, \dots, a_g, b_g, \sigma \, | \\
& (a_1,b_1) \cdots (a_g,b_g), \sigma c \sigma^{-1}= \phi_*(c) \quad \forall c \in \pi_1(S_g,*) \rangle \, ,
\end{alignat*}
where generator $\sigma$ is the class of a section $\sigma : S^1 \to M_\phi$, and $\phi_*$ is the morphism induced by the diffeomorphism $\phi$ in $\pi_1(S_g,*)$.
\end{enumerate} 
\end{prop}

Note that property (ii) allows us to consider the mapping tori as defined by elements of the mapping class group of $S_g$, and change the representative diffeomorphism for the class. 

Let us bring these mapping tori to the symplectic category:

\begin{prop} (Thurston)
The symplectized mapping torus of a diffeomorphism $\phi : S_g \to S_g$ is $X_\phi = M_\phi \times S^1$ is a closed, connected 4--manifold admitting a symplectic structure $\omega$ such that the natural projection 
$p \times \text{Id} : X_\phi \to \bT^2=S^1 \times S^1$ is a symplectic fibration.
\end{prop}

\begin{proof}
The manifolds $X_\phi$ are oriented closed, connected 4--manifolds by their definition. 

Select now an area form $\omega_S$ on $S_g$. There is a diffeomorphism in the mapping class of $\phi$ which preserves $\omega_S$, so we can assume that this is our selected $\phi$. Now, the fact that $\phi^* \omega_S= \omega_S$ allows us to define a 2--form $\omega_0$ in $M_\phi$ by descending $\omega_S$ from $S_g \times [0,1]$ to the quotient defining $M_\phi$. If we denote $p_1 : M_\phi \times S^1 \to M_\phi$, and $\omega_{\bT^2}$ the standard symplectic form in $\bT^2=S^1 \times S^1$, then Thurston's trick tells us that 
$$
\omega_K = p_1^* \omega_0 + K (p \times \text{Id})^* \omega_{\bT^2}
$$
is a symplectic form on $X_\phi$, making $p \times \text{Id}$ a symplectic fibration, for large enough $K$.
\end{proof}

\bigskip

We will be interested in the mapping tori of diffeomorphisms $\phi : S_g \to S_g$ belonging to the Torelli group, and particularly in their comparison to the case $\phi=\text{Id}$, for which $X_{\text{Id}}= S_g \times \bT^2$, $\pi_1(X_{\text{Id}},*) \cong \pi_1(S_g,*) \times \bZ^2$, $\pi_1(X_{\text{Id}},*)/ \pi_1(X_{\text{Id}},*)_l \cong \pi_1(S_g,*)/ \pi_1(S_g,*)_l \times \bZ^2$, and $\cL_{l}(\pi_1(X_{\text{Id}},*)) \cong \cL_{l}(\pi_1(S_g,*)) \oplus k^2$. 

\begin{prop} \label{p:unip}
Let $\phi \in \cT_g$ such that $\phi_*=\text{Id}$ in the nilpotent quotient $\pi_1(S_g,*)/ \pi_1(S_g,*)_n$. Then
\begin{enumerate}
\item If $n \ge 3$, there is an isomorphism of cohomology algebras $H^*(X_\phi; \bZ) \cong H^*(S_g \times \bT^2; \bZ)$.
\item There are isomorphisms for $l \le n$:
$$
\pi_1(X_\phi,*)/ \pi_1(X_\phi,*)_l \cong \pi_1(S_g \times \bT^2,*)/ \pi_1(S_g \times \bT^2,*)_l
$$
and 
$$
\cL_{l-1}(\pi_1(X_\phi,*)) \cong \cL_{l-1}(\pi_1(S_g \times \bT^2))
$$ 
\item If moreover $\phi_* \neq \text{Id}$ in $\pi_1(S_g,*)/ \pi_1(S_g,*)_{n+1}$, then for some $N \ge n+1$
$$
\pi_1(X_\phi,*)/ \pi_1(X_\phi,*)_N \not\cong \pi_1(S_g \times \bT^2,*)/ \pi_1(S_g \times \bT^2,*)_N
$$
and 
$$
\cL_{N}(\pi_1(X_\phi,*)) \not\cong \cL_{N}(\pi_1(S_g \times \bT^2))
$$
\end{enumerate}
\end{prop}
 
\begin{proof}
(ii) If the induced isomorphism $\phi_*$ is the identity in $\pi_1(S_g,*)/ \pi_1(S_g,*)_n$, then the presentation of $\pi_1(M_\phi,*)$ in Prop. \ref{p:mt} (iv) is exactly the same for $\phi, \text{Id}$ modulo brackets of order $n$, so it readily yields for $l \ge n$ isomorphisms $\pi_1(M_\phi,*)/ \pi_1(M_\phi,*)_l \cong \pi_1(M_{\text{Id}},*)/ \pi_1(M_{\text{Id}},*)_l$, where $M_{\text{Id}} = S_g \times S^1$. This isomorphism of nilpotent quotients is preserved by multiplication of the mapping tori by $S^1$, carries over to the unipotent completions $\pi_1(X,*)/ \pi_1(X,*)_l \otimes k$ and to their nilpotent Malcev algebras $\cL_{l-1}(\pi_1(X,*),k)$.

(iii) The presentations found in \cite{Lab} for the graded Lie algebras $\text{gr}\, (\pi_1(S_g,*))$ mean that the center of the unipotent completion $\pi_1(M_{\text{Id}},*) \otimes k$ is the 1--parameter subgroup generated by $\sigma$. On the other hand, if $\phi_* \neq \text{Id}$ modulo $\pi_1(S_g,*)_{n+1}$ then $\phi_* \neq \text{Id}$ modulo $\pi_1(S_g,*)_N$ for all $N \ge n+1$, thus the unipotent completion $\pi_1(M_\phi,*) \otimes k$ has trivial center. Therefore the unipotent completions $\pi_1(M_{\text{Id}},*) \otimes k, \pi_1(M_\phi,*) \otimes k$ cannot be isomorphic, which means that $\pi_1(M_{\text{Id}},*) /\pi_1(M_{\text{Id}},*)_N \otimes k, \pi_1(M_\phi,*) /\pi_1(M_\phi,*)_N \otimes k$ stop being isomorphic at some $N>n$. This non--isomorphism carries over to the symplectized tori $X_\phi, X_{\text{Id}}= S_g \times \bT^2$, e.g. because the center of the pro--unipotent completion of $\pi_1$ is 1--dimensional in $X_\phi$ and 2--dimensional in $S_g \times \bT^2$. 

(i) If $\phi \in \cT_g$ the monodromy of the mapping torus in $H^*(S_g; \bZ)$ is trivial, so by the Leray--Hirsch theorem there is an isomorphism of graded groups 
\begin{equation} \label{eq:isgr}
H^*(M_\phi; \bZ) \cong H^*(S^1; \bZ) \otimes H^*(S_g; \bZ) \, . 
\end{equation}
If, moreover, the induced isomorphism $\phi_*$ is the identity in $\pi_1(S_g,*)/ \pi_1(S_g,*)_3$, then the presentation of $\pi_1(M_\phi,*)$ in Prop. \ref{p:mt} (iv) readily yields an isomorphism $\pi_1(M_\phi,*)/ \pi_1(M_\phi,*)_3 \cong \pi_1(M_{\text{Id}},*)/ \pi_1(M_{\text{Id}},*)_3$, where $M_{\text{Id}} = S_g \times S^1$. 

Let us recall now that for any topological space $M$ there are isomorphisms $H^1(M; \bZ) \cong \text{Hom}\, \left( \pi_1(M,*)/ \pi_1(M,*)_2, \bZ \right)$ and 
$$
\ker \left( \cup : \wedge^2 H^1(M; \bZ) \rightarrow H^2(M; \bZ) \right) \cong \text{Hom}\, \left( \pi_1(M,*)_2/ \pi_1(M,*)_3, \bZ \right)
$$ 
given by the Lyndon/Hochschild--Serre spectral sequence of the action of $\pi_1(M,*)$ on itself by inner automorphisms (see \cite{Wei}). This means that the graded group isomorphisms of (\ref{eq:isgr}) extend to isomorphisms of cup products making commutative the diagram
$$
\begin{array}{ccc}
\wedge^2 H^1(M_\phi; \bZ) & \longrightarrow & H^2(M_\phi; \bZ) \\
\downarrow &                                & \downarrow \\
\wedge^2 H^1(M_{\text{Id}}; \bZ) & \longrightarrow & H^2(M_{\text{Id}}; \bZ)
\end{array}
$$
Poincar\'e duality completes the isomorphism of cohomology algebras between $M_\phi$ and $M_{\text{Id}} = S_g \times S^1$, and the K\"unneth formula extends the isomorphism to $X_\phi, S_g \times \bT^2$.
 \end{proof}

The symplectized Torelli mapping tori have a topology which is very different from that of K\"ahler families of compact Riemann surfaces fibered over an elliptic curve, as the latter are suspensions with nontrivial monodromy in the cohomology of the fibers (\cite{AMN}). There is a stronger obstruction to endowing $X_\phi$ with any holomorphic structure:

\begin{thm} \label{t:main}
Let $X_\phi$ be the symplectized mapping torus for $\phi : S_g \to S_g$ such that $\phi_*=\text{Id}$ in $\pi_1(S_g,*)/ \pi_1(S_g,*)_n$ for $n \ge 3$, but $\phi_* \neq \text{Id}$ in the mapping class group $M(g,0)$. Then $X_\phi$ is $(1,n)$--formal, but it is not 1--formal.
\end{thm}

\begin{proof}
The unipotent completions $\pi_1(X_\phi,*)/ \pi_1(X_\phi,*)_l \otimes k$ are isomorphic to those of $X_{\text{Id}} = S_g \times \bT^2$ for $l \le n$. The latter manifold is the product of two compact connected Riemann surfaces, thus is formal (\cite{DGMS}) and its $(1,l)$--minimal models are those of its cohomology algebra $H^*(S_g \times \bT^2;k)$. As there is an isomorphism of cohomology algebras $H^*(S_g \times \bT^2; k) \cong H^*(X_\phi,k)$, this means that for $l \le n$ the $(1,l)$--minimal models of $X_\phi$ are those of its cohomology algebra, i.e. $X_\phi$ is $(1,n)$--formal.

If $\phi \neq \text{Id}$ in $\cT_g$, then for some $M>n$ it must happen that $\phi_* \neq \text{Id}$ on $\pi_1(S_g,*)/ \pi_1(S_g,*)_M$. Choose as $M$ the first such value. Then, by Prop. \ref{p:unip} (iii), for some $N \ge M$ the Malcev algebra $\cL_N(\pi_1(X_\phi,*),k)$ is not isomorphic to that of $S_g \times \bT^2$. But the cohomology algebra of $X_\phi$ is isomorphic to that of $S_g \times \bT^2$, so the $(1,N)$--minimal model of $H^*(X_\phi;k)$ is dual to the Malcev algebra of $S_g \times \bT^2$, not to the Malcev algebra of $X_\phi$. Therefore $X_\phi$ is not $(1,N)$--formal, and also not 1--formal.
\end{proof}

The consequences of Thm. \ref{t:main} for Massey products of 1--cohomology classes, for whose definition we follow \cite{Dwy}, are:

\begin{cor}
Let $X_\phi$ be the symplectized mapping torus for $\phi : S_g \to S_g$ such that $\phi_*=\text{Id}$ in $\pi_1(S_g,*)/ \pi_1(S_g,*)_n$ for $n \ge 3$, but $\phi_* \neq \text{Id}$ in the mapping class group $M(g,0)$. Then all Massey products of 1--cohomology classes vanish up to length $n$, but there exists a nonvanishing one of length $N>n$.
\end{cor}

\begin{proof}
Massey products of 1--cohomology classes up to length $n$ can be computed in Sullivan's $(1,n)$--minimal model of $X_\phi$, hence in the cohomology ring $H^*(X_\phi;k)$ as $X_\phi$ is $(1,n)$--formal, so they vanish.

Let now $N$ be the first nilpotency class for which $\cL_N(\pi_1(X_\phi,*),k) \not\cong \cL_N(\pi_1(X_{\text{Id}},*),k) \cong 
\cL_N(\pi_1(S_g,*),k) \oplus k^2$. Projection modulo brackets of order $N$, the isomorphism of unipotent completions of class $N-1$ between $X_\phi,X_{\text{Id}}$, and taking quotient by the abelian factor $k^2$ define a surjective group morphism $P$ indicated in the diagram
\begin{equation} \label{eq:massey}
\begin{array}{ccc}
 & & \pi_1(S_g,*)/ \pi_1(S_g,*)_{N+1} \otimes k \\
 & & \downarrow \\
\pi_1(X_\phi,*)/ \pi_1(X_\phi,*)_{N+1} & \stackrel{P}{\longrightarrow} & \pi_1(S_g,*)/ \pi_1(S_g,*)_N \otimes k
\end{array}
\end{equation}
sending every generator $a_i,b_j$ in the presentation of $\pi_1(M_\phi,*)$ of Prop. \ref{p:mt} to its identically named generator in the standard presentation of $\pi_1(S_g,*)$, and sending $\sigma$ and the loop from the additional factor $S^1$ to 1. We will check that it is impossible to lift the morphism $P$ in \eqref{eq:massey} to a group morphism $\tilde{P} : \pi_1(X_\phi,*)/ \pi_1(X_\phi,*)_{N+1} \rightarrow \pi_1(S_g,*)/ \pi_1(S_g,*)_{N+1} \otimes k$.

The fact that $\phi_* \neq \text{Id}$ on $\pi_1(S_g,*)/ \pi_1(S_g,*)_{N+1}$ implies that for some generator $c =a_i$, or $c=b_j$ in the standard presentation of $\pi_1(S_g,*)$ the element $w_c=\phi_*(c) c^{-1}$ is a bracket of order $N$, nontrivial modulo brackets of order $N+1$ in $\pi_1(S_g,*)$. 

In $\pi_1(X_\phi,*)$ we have that $w_c= (\sigma,c)$. As $P(\sigma)=1$, any lift $\tilde P(\sigma)$ would have to lie in $\pi_1(S_g,*)_N \otimes k$, the center of $\pi_1(S_g,*)/ \pi_1(S_g,*)_{N+1} \otimes k$, thus $\tilde{P}(\sigma,c)=1$.

Regard $w_c$ as a word in the generators of $\pi_1(X_\phi,*)$ given by presentation of Prop. \ref{p:mt}. These generators are sent by $P$ to their namesakes in the standard presentation of $\pi_1(S_g,*)$. Thus any lift $\tilde{P}$ would have to send the generators $a_1,b_1, \dots, a_g,b_g$ to images $a_1 w_{a_1}, \dots b_g w_{b_g}$ respectively, where $w_{a_1}, \dots, w_{b_g}$ are elements in $\pi_1(S_g,*)_N/ \pi_1(S_g,*)_{N+1} \otimes k$. The latter subgroup is the center of $\pi_1(S_g,*)/ \pi_1(S_g,*)_{N+1} \otimes k$, so $\tilde{P}(w_c)$ has to be exactly the word $w_c$, which is not trivial modulo $\pi_1(S_g,*)_{N+1}$, hence not trivial in $\pi_1(S_g,*)/ \pi_1(S_g,*)_{N+1} \otimes k$.
 
The conclusion of the above discussion is that the morphism to a unipotent group $\pi_1(S_g,*)/ \pi_1(S_g,*)_{N+1} \otimes k$ modulo its center, $P : \pi_1(X_\phi,*)/ \pi_1(X_\phi,*)_N \rightarrow \pi_1(S_g,*)/ \pi_1(S_g,*)_{N+1} \otimes k$ does not lift to the unipotent group itself. By the characterization of Dwyer (\cite{Dwy}), a Massey product of 1--cohomology classes in $\pi_1(X_\phi,*)$ must exist and not vanish.
\end{proof} 

\bigskip

We will conclude with another consequence of Thm. \ref{t:main}:  

\begin{cor}
If a smooth map $\phi : S_g \to S_g$ is such that $\phi_*=\text{Id}$ in $\pi_1(S_g,*)/ \pi_1(S_g,*)_n$ for $n \ge 3$, but $\phi_* \neq \text{Id}$ in the mapping class group $M(g,0)$, then its symplectized mapping torus $X_\phi$ does not admit a holomorphic structure.
\end{cor}

\begin{proof}
The first Betti number of the 4--manifold $X_\phi$ is even, and all compact holomorphic manifolds of complex dimension 2 with even $b_1$ were shown by Kodaira to be deformation equivalent, in particular diffeomorphic, to smooth complex projective surfaces (see Ch. IV.3 of \cite{BHPV}). The latter are compact K\"ahler manifolds, which are formal (\cite{DGMS}), in particular 1--formal. So by Thm. \ref{t:main} $X_\phi$ cannot support a holomorphic structure.   
\end{proof}

\end{document}